\documentclass[12pt]{amsart}

\newtheorem{theorem}{Theorem}[section]

\theoremstyle{definition}
\newtheorem{definition}{Definition}[section]
\newtheorem{corollary}{Corollary}[section]

\theoremstyle{remark}

\numberwithin{equation}{section}

\begin{document}
\title[Pseudo parallel contact CR-Submanifolds of Kenmotsu manifolds]{Pseudo parallel contact CR-Submanifolds of Kenmotsu manifolds}
\author[S. K. Hui and P. Mandal]{Shyamal Kumar Hui and Pradip Mandal}
\subjclass[2000]{53C15, 53C42.}
\keywords{Kenmotsu manifold, CR-submanifold, pseudo parallel submanifold.}
\maketitle
\begin{abstract}
The present paper deals with the study of pseudo parallel (in the sense of Chaki and in the sense of Deszcz) contact CR-submanifolds
with respect to Levi-Civita connection as well as semisymmetric metric connection of
Kenmotsu manifolds and prove that these corresponding two classes are equivalent with a certain condition.
\end{abstract}
\section{Introduction}
In \cite{TANNO} Tanno classified connected almost contact metric manifolds whose automorphism groups possess the maximum dimension.
For such a manifold, the sectional curvature of plane sections containing $\xi$ is a constant, say $c$. He proved that they could
be divided into three classes:
(i) homogeneous normal contact Riemannian manifolds with $c > 0$,
(ii) global Riemannian products of a line or a circle with a K\"{a}hler manifold of constant holomorphic sectional curvature
          if $c=0$ and
(iii) a warped product space $\mathbb{R} \times _f \mathbb{C}^n$ if $c< 0$.
It is known that the manifolds of class (i) are characterized by admitting a Sasakian structure.
The manifolds of class (ii) are characterized by a tensorial relation admitting a cosymplectic structure. Kenmotsu \cite{KEN}
characterized the differential geometric properties of the manifolds of class (iii) which are nowadays called Kenmotsu
manifolds and later studied by several authors (\cite{HUI1}-\cite{HUI3}) etc.\\
\indent The contact CR-submanifolds are rich and very much interesting subject. The study of the differential geometry of a contact
CR-submanifolds as a generalization of invariant and anti-invariant submanifolds of an almost contact metric manifold
was initiated by Bejancu \cite{BEJ}. Thereafter several authors studied submanifolds as well as contact CR-submanifolds
of different classes of almost contact metric manifolds such as Chen (\cite{CHEN1}, \cite{CHEN2}), Hasegawa and Mihai \cite{HASE},
Hui et al. \cite{HAN}, Jamali and Shahid \cite{JAMALI}, Khan et al. (\cite{KHAN}, \cite{MAKHAN}), Munteanu \cite{MUN}, Murathan et al. \cite{MURA}
and many others. Also Atceken and his co-author (\cite{ATCE1}, \cite{ATCE2}) studied contact CR-submanifolds of Kenmotsu manifolds.\\
\indent A Riemannian manifold is said to be Ricci symmetric or Ricci parallel if its Ricci tensor $S$ of type $(0,2)$
satisfies $\nabla S=0$, where $\nabla$ denotes the Riemannian connection. During the last six decades, the notion of Ricci
symmetry has been weakened by many authors in several ways to a different
extent such as Ricci recurrent manifolds by Patterson \cite{PATT}, pseudo Ricci symmetric manifolds or pseudo Ricci parallel manifolds by  Chaki \cite{CHAKI}.\\
\indent A non-flat Riemannian manifold $(M,g)$ is said to be pseudo Ricci symmetric \cite{CHAKI} if its Ricci tensor $S$ of type
$(0,2)$ is not identically zero and satisfies the condition
\begin{equation}\label{eqn1.1}
  (\nabla_X S)(Y,Z)=2\alpha(X)S(Y,Z)+\alpha(Y)S(X,Z)+\alpha(Z)S(X,Y)
\end{equation}
for all vector fields $X,Y,Z\in \chi(M)$, where $\alpha$ is a nowhere vanishing $1$-form.
\indent Again in another direction, Szab\'{o} \cite{SZA} generalized the notion of Ricci symmetric manifolds to Ricci semisymmetric manifolds,
 which also generalized by Deszcz \cite{DES}
as Ricci pseudosymmetric manifolds.\\
\indent A Riemannian manifold $(M^n,g)~(n>2)$ is said to be Ricci pseudosymmetric \cite{DES} if and only if
\begin{equation}\label{eqn1.2}
  (R(X,Y)\cdot S)(Z,U)=L_SQ(g,S)(Z,U;X,Y)
\end{equation}
holds on $U_S=\{x\in M:(S-\frac{r}{n}g)_x\neq0\}$ for all $X,Y,Z,U\in \chi(M)$, where $L_S$ is some function on $U_S$, $R$
is the curvature tensor, $S$ is the Ricci tensor and $r$ is the scalar curvature of the manifold $M$.
Here the tensor $Q(g,S)$ is defined as
\begin{eqnarray}
\label{eqn1.3}
  Q(g,S)(Z,U;X,Y) &=& -((X\wedge_gY)\cdot S)(Z,U) \\
  \nonumber&=&         S((X\wedge_gY)Z,U)+S(Z,(X\wedge_gY)U),
\end{eqnarray}
where $(X\wedge_gY)Z=g(Y,Z)X-g(X,Z)Y.$\\
\indent It may be noted that pseudo Ricci symmetric manifolds by Chaki is different from Ricci pseudosymmetric manifolds by Deszcz.\\
\indent The present paper deals with the study of pseudo parallel contact CR-subman-ifolds of Kenmotsu manifolds.
The paper is organized as follows. Section $2$ is concerned with preliminaries. Section $3$ is devoted to the study of pseudo
parallel (in the sense of Chaki) contact CR-submanifolds as well  as pseudo parallel (in the sense of Deszcz) contact CR-submanifolds
 of Kenmotsu manifolds. It is shown that pseudo parallel (in the sense of Deszcz) contact CR-submanifolds and
pseudo parallel (in the sense of Chaki) contact CR-submanifolds of Kenmotsu manifolds are equivalent with a certain condition.
The pseudo parallel (in the sense of Chaki) contact CR-submanifolds with respect to semisymmetric metric connection as well as
pseudo parallel (in the sense of Deszcs) contact CR-submanifolds with respect to semisymmetric metric connection of Kenmotsu
manifolds with respect to semisymmetric metric connection are studied in section $4$ and it is proved that these two classes
are also equivalent with a certain condition.
\section{Preliminaries}
An odd dimensional smooth manifold $(\overline{M}^{2n+1},g)$ is said to be an almost contact metric manifold \cite{BLAIR} if it admits a $(1,1)$
tensor field $\phi$, a vector field $\xi$, an $1$-form $\eta$ and a Riemannian metric $g$ which satisfy
\begin{equation}\label{eqn2.1}
  \phi \xi=0,\ \ \ \eta(\phi X)=0, \ \ \ \phi^2 X=-X+\eta(X)\xi,
\end{equation}
\begin{equation}\label{eqn2.2}
  g(\phi X,Y)=-g(X,\phi Y), \ \ \ \eta(X)=g(X,\xi), \ \ \ \eta(\xi)=1,
\end{equation}
\begin{equation}\label{eqn2.3}
  g(\phi X,\phi Y)=g(X,Y)-\eta(X)\eta(Y)
\end{equation}
for all vector fields $X,Y$ on $M$.\\
\indent An almost contact metric manifold $\overline{M}^{2n+1}(\phi, \xi, \eta, g)$ is said to be Kenmotsu manifold if the following conditions hold \cite{KEN}:
\begin{equation}\label{eqn2.4}
  \overline{\nabla}_X \xi=X-\eta(X)\xi,
\end{equation}
\begin{equation}\label{eqn2.5}
 (\overline{\nabla}_X \phi)(Y)=g(\phi X ,Y)\xi-\eta(Y)\phi X,
\end{equation}
where $\overline{\nabla}$ denotes the Riemannian connection of $g$.\\
\indent In a Kenmotsu manifold, the following relations hold \cite{KEN}:
\begin{equation}\label{eqn2.6}
 (\overline{\nabla}_X \eta)(Y)=g(X ,Y)-\eta(X)\eta(Y),
\end{equation}
\begin{equation}\label{eqn2.7}
  \overline{R}(X,Y)\xi=\eta(X)Y-\eta(Y)X,
\end{equation}
\begin{equation}\label{eqn2.8}
  \overline{R}(\xi,X)Y=\eta(Y)X-g(X,Y)\xi,
\end{equation}
\begin{equation}\label{eqn2.9}
  \overline{S}(X,\xi)=-2n\eta(X)
\end{equation}
for any vector field $X,Y$ on $\overline{M}$ and $\overline{R}$ is the Riemannian curvature tensor and $\overline{S}$ is the Ricci tensor of type
$(0,2)$.\\
\indent Let $M$ be a $(2m+1)$-dimensional $(m<n)$ submanifold of a Kenmotsu manifold $\overline{M}$. Throughout the paper we assume that
the submanifold $M$ of $\overline{M}$ is tangent to the structure vector field $\xi$.\\
\indent Let $\nabla$ and $\nabla ^\bot$ are the induced connections on the tangent bundle $TM$ and the normal bundle $T^\bot M$ of $M$ respectively.
Then the Gauss and Weingarten formulae are given by
\begin{equation}\label{eqn2.10}
  \overline{\nabla}_X Y=\nabla_X Y+h(X,Y)
\end{equation}
and
\begin{equation}\label{eqn2.11}
  \overline{\nabla}_XV=-A_V X+\nabla _X^ {\bot}V
\end{equation}
for all $X,Y\in \Gamma(TM)$ and $V\in \Gamma(T^\bot M)$, where $h$ and $A_V$ are second fundamental form and the shape operator
(corresponding to the normal vector field $V$) respectively for the immersion of $M$ into $\overline{M}$. The second fundamental form
$h$ and the shape operator $A_V$ are related by
\begin{equation}\label{eqn2.12}
  g(h(X,Y),V)=g(A_V X,Y)
\end{equation}
for any $X,Y\in \Gamma(TM)$ and $V\in \Gamma(T^\bot M)$, where g is the Riemannian metric on $\overline{M}$ as well as on $M$.\\
\indent For any submanifold $M$ of a Riemannian manifold $\overline{M}$, the equation of Gauss is given by
\begin{eqnarray}
\label{eqn2.13}
  \overline{R}(X,Y)Z &=& R(X,Y)Z+A_{h(X,Z)}Y-A_{h(Y,Z)}X\\
  \nonumber&+&(\overline{\nabla}_X h)(Y,Z)-(\overline{\nabla}_Y h)(X,Z)
\end{eqnarray}
for any $X,Y,Z\in \Gamma(TM)$, where $\overline{R}$ and $R$ denote the the Riemannian curvature tensors of $\overline{M} $ and $M$ respectively.
The covariant derivative $\overline{\nabla}h$ of $h$ is defined by
\begin{equation}\label{eqn2.14}
  (\overline{\nabla}_X h)(Y,Z)=\nabla_X^\bot h(Y,Z)-h(\nabla_X Y,Z)-h(\nabla_XZ,Y).
\end{equation}
The normal part $(\overline{R}(X,Y)Z)^\bot$ of $\overline{R}(X,Y)Z$ from $(\ref{eqn2.13}) $ is given by
\begin{equation}\label{eqn2.15}
  (\overline{R}(X,Y)Z)^\bot=(\overline{\nabla}_Xh)(Y,Z)-(\overline{\nabla}_Yh)(X,Z),
\end{equation}
which is known as Codazzi equation. In particular, if $(\overline{R}(X,Y)Z)^\bot=0$ then $M$ is said to be curvature-invariant submanifold of
$\overline{M}$.\\
\indent On the other hand, since $M$ is tangent to $\xi $, we have
\begin{equation}\label{eqn2.16}
  A_V \xi=h(X,\xi)=0.
\end{equation}
In view of (\ref{eqn2.16}) we have from (\ref{eqn2.4}) and (\ref{eqn2.10}) that
\begin{equation}\label{eqn2.17}
  \nabla_X \xi =\overline{\nabla}_X \xi=X-\eta(X)\xi.
\end{equation}
Also in view of (\ref{eqn2.16}) we have from (\ref{eqn2.7}) and (\ref{eqn2.13}) that
\begin{equation}\label{eqn2.18}
  R(X,Y)\xi=\overline{R}(X,Y)\xi=\eta(X)Y-\eta(Y)X.
\end{equation}
\indent Let $M$ be an isometrically immersed in an almost contact metric manifold $\overline{M}$ then for every $p\in M$,
there exist a maximal invariant subspace denoted by $D_p$ of the tangent space $T_pM$ of $M$. If the dimension of $D_p$ is the same
for all value of $p\in M$, then $D_p$ gives an invariant distribution $D$ on $M$.
\begin{definition}\cite{ATCE2}
  Let $M$ be an isometrically immersed submanifold of a Kenmotsu manifold $\overline{M}$. Then $M$ is called a contact
CR-submanifold of $\overline{M}$ if there is a differential distribution $D:p\rightarrow D_p\subseteq T_p M$ on $M$ satisfying the
  following conditions:\\
(i) $\xi \in D$,\\
(ii) D is invariant with respect to $\phi$, i.e., $\phi(D_p)\subseteq D_p$ for each $p\in M$ and\\
(iii) the orthogonal complementary distribution $D^\bot:p\rightarrow D_p^\bot{\subseteq T_pM}$ satisfying $\phi(D^\bot _p)\subseteq T_p^\bot M$
for each $p\in M$.
\end{definition}
\indent A contact CR-submanifold is called anti-invariant (or totally real) if $D_p=\{0\}$ and invariant (or holomorphic )
if $D_p^\bot=\{0\}$, respectively for any $p\in M$. It is called proper contact CR-submanifold if neither $D_p=\{0\}$ nor $D_p^\bot=\{0\}$.\\
\indent In \cite{FRIED} Friedmann and Schouten introduced the notion
of semisymmetric linear connection on a differentiable manifold. Then in 1932 Hayden \cite{HAYDEN}
introduced the idea of metric connection with torsion on a
Riemannian manifold. A systematic study of the semisymmetric metric connection on
a Riemannian manifold has been given by Yano in 1970 \cite{YANO5}.\\
\indent A linear connection of a Kenmotsu manifold $\overline{M}$
is said to be a semisymmetric connection  if its torsion tensor
$\tau$ of the connection $\widetilde{\overline{\nabla}}$ is of the form
\begin{equation}
\label{eqn1.1a}
\tau(X,Y) = \widetilde{\overline{\nabla}}_{X}Y - \widetilde{\overline{\nabla}}_{Y}X - [X,Y]
\end{equation}
satisfies
\begin{equation}
\label{eqn1.2a}
\tau(X,Y)=\eta(Y) X - \eta(X) Y,
\end{equation}
where $\eta$ is a $1$-form. Again, if the semisymmetric connection $\widetilde{\overline{\nabla}}$
satisfies the condition
\begin{equation}
\label{eqn1.3a}
(\widetilde{\overline{\nabla}}_{X}g)(Y,Z) = 0
\end{equation}
for all $X$, $Y$, $Z\in\chi(\overline{M})$, where $\chi(\overline{M})$ is the Lie algebra of vector fields on the manifold $\overline{M}$, then
$\widetilde{\overline{\nabla}}$ is said to be a semisymmetric metric connection. Semisymmetric metric connection have
been studied by many authors in several ways to a different extent. In this connection it is mentioned that Sular
\cite{SULAR} and Tripathi \cite{TRIPATHI} studied semisymmetric metric connection in Kenmotsu manifold.\\
\indent Let $\overline{M}$ be an $n$-dimensional Kenmotsu manifold and $\overline{\nabla}$ be the Levi-Civita connection on $\overline{M}$.
A semisymmetric metric connection $\widetilde{\overline{\nabla}}$ in a Kenmotsu  manifold is defined by (\cite{SULAR}, \cite{TRIPATHI})
\begin{equation}
\label{eqn2.6a}
\widetilde{\overline{\nabla}}_{X}Y = \overline{\nabla}_{X}Y + \eta(Y)X - g(X,Y)\xi.
\end{equation}
\indent If $\overline{R}$ and $\widetilde{\overline{R}}$ are respectively the curvature tensor with respect to Levi-Civita connection $\overline{\nabla}$
and semisymmetric metric connection $\widetilde{\overline{\nabla}}$ in a Kenmotsu manifold then we have \cite{SULAR}
\begin{eqnarray}
\label{eqn2.7a}
\widetilde{\overline{R}}(X,Y)Z&=&\overline{R}(X,Y)Z - 3\{g(Y,Z)X - g(X,Z)Y\}\\
\nonumber&+&2\{\eta(Y)\eta(Z)X - \eta(X)\eta(Z)Y\}\\
\nonumber&+&2\{g(Y,Z)\eta(X) - g(X,Z)\eta(Y)\}\xi.
\end{eqnarray}
Also from (\ref{eqn2.7a}) we have
\begin{equation}
\label{eqn2.8a}
\widetilde{\overline{S}}(Y,Z)=\overline{S}(Y,Z) - (3n-5)g(Y,Z) + 2(n-2)\eta(Y)\eta(Z),
\end{equation}
where $\widetilde{\overline{S}}$ and $\overline{S}$ are respectively the Ricci tensor of a Kenmotsu manifold\\ $M^{n}(\phi, \xi, \eta, g)$
with respect to semisymmetric metric connection $\widetilde{\overline{\nabla}}$ and Levi-Civita connection $\overline{\nabla}$.\\
Again from (\ref{eqn2.7a}) we get
\begin{equation}\label{eqn2.25}
  \widetilde{\overline{R}}(X,Y)\xi=2\{\eta(X)Y-\eta(Y)X\}.
\end{equation}
\section{pseudo parallel contact cr-submanifolds of kenmotsu manifolds}
This section deals with the study of pseudo parallel (in the sense of Chaki) contact CR-submanifolds and pseudo parallel
(in the sense of Deszcz) contact CR-submanifolds of Kenmotsu manifolds.\\
\indent Following the definition of pseudo Ricci symmetric manifold \cite{CHAKI} in the sense of Chaki, we can define the following:
\begin{definition}
  A submanifold $M$ of a Kenmotsu manifold $\overline{M} $ is called pseudo parallel in the sense of Chaki if its second fundamental
  form $h$ satisfies
  \begin{equation}\label{eqn3.1}
    (\nabla_Xh)(Y,Z)=2\alpha(X)h(Y,Z)+\alpha(Y)h(X,Z)+\alpha(Z)h(X,Y)
     \end{equation}
for all $X,Y,Z $ on $M$, where $\alpha$ is a nowhere vanishing $1$-form.
\end{definition}
In particular if $\alpha(X)=0$ then $h$ is said to be parallel and $M$ is said to be parallel submanifold of $\overline{M}$.\\
We now prove the following:
\begin{theorem}
  Let $M $ be a contact CR-submanifold of a Kenmotsu manifold $\overline{M}$. Then $M$ is totally geodesic if and only if $M$
  is pseudo parallel in the sense of Chaki with $\alpha(\xi)\neq-1$.
\end{theorem}
\begin{proof}
  Suppose that $M$ is a contact CR-submanifold of a Kenmotsu manifold $\overline{M}$ such that $M$ is pseudo parallel in the sense of Chaki.
  Then by virtue of $(\ref{eqn2.14})$ we have from $(\ref{eqn3.1})$ that
  \begin{align}\label{eqn3.2}
 &   \nabla_X^\bot h(Y,Z)-h(\nabla_X Y,Z)-h(Y,\nabla_XZ) & \\
\nonumber&=  2\alpha(X)h(Y,Z)+\alpha(Y)h(X,Z)+\alpha(Z)h(X,Y) .
  \end{align}
  Putting $Z=\xi $ in $(\ref{eqn3.2})$ and using $(\ref{eqn2.16})$ we get
\begin{equation}\label{eqn3.3}
 - h(Y,\nabla_X \xi)=\alpha(\xi)h(X,Y).
\end{equation}
In view of $(\ref{eqn2.16})$ and $(\ref{eqn2.17})$, $(\ref{eqn3.3})$ yields
$$[1+\alpha(\xi)]h(X,Y)=0,$$
which implies that $h(X,Y)=0$ for all $X,Y$ on $M $ as $\alpha(\xi)\neq -1$.\\
Hence $M$ is totally geodesic submanifold. The converse part is trivial. This proves the theorem.
\end{proof}
\begin{corollary}\cite{ATCE2}
  Let $M$ be a contact CR-submanifold of a Kenmotsu manifold $\overline{M}$. Then $M$ is totally geodesic if and only if $M$ is parallel.
\end{corollary}
Again following the definition of Ricci pseudosymmetric manifold in the sense of Deszcz, we can define the following:
\begin{definition}
  A submanifold $M$ of Kenmotsu manifold $\overline{M}$ is said to be pseudo parallel in the sense of Deszcz if its second fundamental
   form $h$  satisfies
   \begin{eqnarray}
   \label{eqn3.4}
     \overline{R}(X,Y)\cdot h &=& (\overline{\nabla}_X\overline{\nabla}_Y-\overline{\nabla}_Y\overline{\nabla}_X-\overline{\nabla}_{[X,Y]})h \\
     \nonumber &=& L_1Q(g,h)
   \end{eqnarray}
for all vector fields $X,Y$ tangent to $M$, where $\overline{R}$ is the curvature tensor of $\overline{M}$. In particular, if $L_1=0$ then
$M$ is said to be semiparallel.
\end{definition}
We now prove the following:
\begin{theorem}
  Let $M$ be a contact CR-submanifold of a Kenmotsu manifold $\overline M$. Then $M$ is totally geodesic if and only if $M$ is pseudo parallel
  in the sense of Deszcz with $L_1\neq 1$.
\end{theorem}
\begin{proof}
  Let $M$ be a contact CR-submanifold of a Kenmotsu manifold $\overline{M}$.\\
  First, suppose that $M$ is pseudo parallel in the sense of Deszcz.
  Then we have the relation $(\ref{eqn3.4})$, i.e.,
\begin{equation}\label{eqn3.5}
  (\overline{R}(X,Y)\cdot h)(Z,U)=L_1Q(g,h)(Z,U;X,Y).
\end{equation}
It is known from tensor algebra that
\begin{eqnarray}
\label{eqn3.6}
(\overline{R}(X,Y)\cdot h)(Z,U)&=&R^\bot(X,Y)h(Z,U)-h(R(X,Y)Z,U)\\
\nonumber&-&h(Z,R(X,Y)U)
\end{eqnarray}
for all vector fields $X,Y,Z$ and $U$, where
\begin{equation*}
  R^\bot(X,Y)=[\nabla_X^\bot,\nabla_Y^\bot]-\nabla_{[X,Y]}^\bot.
\end{equation*}
By similar way of $(\ref{eqn1.3})$, we have
\begin{eqnarray}
\label{eqn3.7}
  Q(g,h)(Z,U;X,Y) &=& g(Y,Z)h(X,U)-g(X,Z)h(Y,U)\\
  \nonumber&+& g(Y,U)h(X,Z)-g(X,U)h(Y,Z).
\end{eqnarray}
In view of $(\ref{eqn3.6})$ and $(\ref{eqn3.7})$ we get from $(\ref{eqn3.5})$ that
\begin{align}\label{eqn3.8}
 & R^\bot(X,Y)h(Z,U)-h(R(X,Y)Z,U)-h(Z,R(X,Y)U)&\\
 \nonumber&=L_1[g(Y,Z)h(X,U)-g(X,Z)h(Y,U)+g(Y,U)h(X,Z)&\\
 \nonumber&-g(X,U)h(Y,Z)].
\end{align}
Putting $X=U=\xi$ in $(\ref{eqn3.8})$ and using $(\ref{eqn2.16})$ we get
\begin{equation}\label{eqn3.9}
  h(Z,R(\xi,Y)\xi)=L_1h(Y,Z).
\end{equation}
Feeding $(\ref{eqn2.18})$ in $(\ref{eqn3.9})$ and using $(\ref{eqn2.16})$ we get
  $(L_1-1)h(Y,Z)=0$,
which implies that $h(Y,Z)=0$ for all $Y,Z$ on $M$, i.e., $M$ is totally geodesic, since $L_1\neq 1$. The converse part is trivial.
This proves the theorem.
\end{proof}
\begin{corollary}
  Let $M$ be a contact CR-submanifold of a Kenmotsu manifold $\overline{M}$. Then $M$ is totally geodesic if and only if $M$
  is semiparallel.
\end{corollary}
From Corollary 3.1, Corollary 3.2, Theorem 3.1, Theorem 3.2, we can state the following:
\begin{theorem}
  Let $M$ be a contact CR-submanifold of a Kenmotsu manifold $\overline{M}$. Then the following statements are equivalent: \\
\emph{(i)} $M$ is totally geodesic,\\
\emph{(ii)} $M$ is parallel,\\
\emph{(iii)} $M$ is semiparallel,\\
\emph{(iv)} $M$ is pseudo parallel in the sense of Chaki with $\alpha(\xi)\neq -1$,\\
\emph{(v)} $M$ is pseudo parallel in the sense of Deszcz with $L_1\neq 1$.
\end{theorem}
\section{pseudo parallel contact cr-submanifolds of kenmotsu manifolds with respect to semisymmetric metric connection}
We now consider $M$ be a contact CR-submanifold of a Kenmotsu manifold $\overline{M}$ with respect to Levi-Civita connection $\overline{\nabla}$
and semisymmetric metric connection  $\widetilde{\overline{\nabla}}$. Let $\nabla$ be the induced connection on $M$ from the connection
$\overline{\nabla}$ and $\widetilde{\nabla}$ be the induced connection on $M$ from the connection $\widetilde{\overline{\nabla}}$.\\
\indent Let $h$ and $\widetilde{h}$ be the second fundamental form with respect to Levi-Civita connection and semisymmetric metric connection,
respectively. Then we have
\begin{equation}\label{eqn3.10}
  \widetilde{\overline{\nabla}}_X Y=\widetilde{\nabla}_X Y+\widetilde{h}(X,Y).
\end{equation}
 By virtue of (\ref{eqn2.10}) and (\ref{eqn2.6a}) we have from (\ref{eqn3.10}) that
\begin{eqnarray}
\label{eqn3.11}
  \widetilde{\nabla}_X Y+\widetilde{h}(X,Y) &=& \overline{\nabla}_XY+\eta(Y)X-g(X,Y)\xi \\
\nonumber  &=& \nabla_XY+h(X,Y)+\eta(Y)X-g(X,Y)\xi.
\end{eqnarray}
Since $X,\xi\in TM$, by equating the tangential and normal components of (\ref{eqn3.11}) we get
\begin{equation}\label{eqn3.12}
  \widetilde{\nabla}_X Y=\nabla_XY+\eta(Y)X-g(X,Y)\xi
\end{equation}
and
\begin{equation}\label{eqn3.13}
  \widetilde{h}(X,Y)=h(X,Y),
\end{equation}
which implies that the second fundamental forms with respect to Levi-Civita connection and semisymmetric metric connection are same.\\
\indent This leads to the following:
\begin{theorem}
  Let $M$ be a contact CR-submanifold of a Kenmotsu manifold $\overline{M}$ with respect to semisymmetric metric connection.
  Then\\
 \emph{(i)} $M$ admits semisymmetric metric connection.\\
 \emph{(ii)} The second fundamental forms with respect to Riemannian connection and semisymmetric metric connection are equal.
\end{theorem}
Now we define the following:
\begin{definition}
  A submanifold $M$ of a Kenmotsu manifold $\overline{M}$ is called pseudo parallel in the sense of Chaki if its second
  fundamental form $\widetilde{h}$ satisfies
  \begin{equation*}
    (\widetilde{\nabla}_X\widetilde{h})(Y,Z)=2\alpha(X)\widetilde{h}(Y,Z)+\alpha(Y)\widetilde{h}(X,Z)+\alpha(Z)\widetilde{h}(X,Y)
  \end{equation*}
for all $X,Y,Z$ on $M$.\\
\indent Let us take $M$ be a contact CR-submanifold of a Kenmotsu manifold with respect to semisymmetric metric connection.
  Suppose that $M$ is pseudo parallel in the sense of Chaki with respect to semisymmetric metric connection. Then we have 
\begin{equation}\label{eqn4.4a}
(\widetilde{\nabla}_Xh)(Y,Z)=2\alpha(X)h(Y,Z)+\alpha(Y)h(X,Z)+\alpha(Z)h(X,Y).
\end{equation}
\end{definition}
In view of (\ref{eqn3.12}) and (\ref{eqn2.16}) we have from (\ref{eqn4.4a}) that
\begin{align*}
  &(\nabla_Xh)(Y,Z)+g(h(Y,Z),\xi)-g(X,h(Y,Z))\xi &\\
 &-\eta(Y)h(X,Z)-\eta(Z)h(X,Y)&\\
   &= 2\alpha(X)h(Y,Z)+\alpha(Y)h(X,Z)+\alpha(Z)h(X,Y),
\end{align*}
i.e.,
\begin{align}\label{eqn4.4b}
 & \nabla^\bot_X h(Y,Z)-h(\nabla_XY,Z)-h(Y,\nabla_XZ)&\\
\nonumber & +g(h(Y,Z),\xi)-g(X,h(Y,Z))\xi &\\
\nonumber & -\eta(Y)h(X,Z)-\eta(Z)h(X,Y)&\\
\nonumber &= 2\alpha(X)h(Y,Z)+\alpha(Y)h(X,Z)+\alpha(Z)h(X,Y).
\end{align}
Putting $Z=\xi$ in (\ref{eqn4.4b}) and using (\ref{eqn2.16}), we get
\begin{equation}\label{eqn4.4c}
  -h(Y,\nabla_X\xi)-h(X,Y)=\alpha(\xi) h(X,Y).
\end{equation}
 By virtue of (\ref{eqn2.16}) and (\ref{eqn2.17}) we have from (\ref{eqn4.4c}) that
 $[\alpha(\xi)+2]h(X,Y)=0,$ which implies that $h(X,Y)=0$ provided $\alpha(\xi)\neq -2.$\\
\indent Thus we can state the following:
\begin{theorem}
  Let $M$ be a contact CR-submanifold of a Kenmotsu manifold $\overline{M}$ with respect to semisymmetric metric connection.
Then $M$ is totally geodesic if and only if $M$ is pseudo parallel with respect to semisymmetric metric connection in the sense
of Chaki, provided $\alpha(\xi)\neq -2$.
\end{theorem}
\begin{corollary}
  Let $M$ be a contact CR-submanifold of a Kenmotsu manifold $\overline{M}$ with respect to semisymmetric metric connection.
  Then $M$ is totally geodesic if and only if $M$ is parallel with respect to semisymmetric metric connection.
\end{corollary}
\begin{definition}
  A submanifold $M$ of a Kenmotsu manifold $\overline{M}$ with respect to semisymmetric metric connection is said to be pseudo parallel
  in the sense of Deszcz with respect to semisymmetric metric connection if
\begin{equation}
\label{eqn3.14}
\widetilde{\overline{R}}(X,Y)\cdot \widetilde{h}=L_1Q(g,\widetilde{h})
\end{equation}
holds for all vector fields $X,Y$ tangent to $M$, where $\widetilde{\overline{R}}$ is the curvature tensor of $\overline{M}$.
In particular if $L_1=0$ then $M$ is said to be semiparallel with respect to semisymmetric metric connection.
\end{definition}
We now prove the following:
\begin{theorem}
  Let $M$ be a contact CR-submanifold of a Kenmotsu manifold $\overline{M} $ with respect to semisymmetric metric connection. Then
  $M$ is totally geodesic if and only if $M$ is pseudo parallel in the sense of Deszcz with respect to semisymmetric metric connection
  with $L_1\neq 2$.
\end{theorem}
\begin{proof}
Let $M$ be a contact CR-submanifold of a Kenmotsu manifold $\overline{M}$ with respect to semisymmetric metric connection.
Suppose that $M$ is pseudo parallel in the sense of Deszcz with respect to semisymmetric metric connection. Then we have
from $(\ref{eqn3.14})$ that
\begin{equation}
\label{eqn3.15}
\widetilde{\overline{R}}(X,Y)\cdot h=L_1Q(g,h),
\end{equation}
i.e.
\begin{align}\label{eqn3.16}
&\widetilde{R}^\bot(X,Y) h(Z,U)-h(\widetilde{R}(X,Y)Z,U)-h(Z,\widetilde{R}(X,Y)U) &  \\
\nonumber  &=L_1[g(Y,Z)h(X,U)-g(X,Z)h(Y,U)+g(Y,U)h(X,Z) & \\
\nonumber &\ \ \ -g(X,U)h(Y,Z)].
\end{align}
Putting $X=U=\xi$ in (\ref{eqn3.16}) and using (\ref{eqn2.16}), we get
\begin{equation}\label{eqn3.17}
  h(Z,\widetilde{R}(\xi, Y)\xi)=L_1h(Y,Z).
\end{equation}
Now by virtue of (\ref{eqn2.16}) and (\ref{eqn2.25}) we have
\begin{eqnarray}
\label{eqn3.18}
  \widetilde{R}(X,Y)\xi &=& \widetilde{\overline{R}}(X,Y)\xi\\
  \nonumber&=&2\{\eta(X)Y-\eta(Y)X\}.
\end{eqnarray}
In view of (\ref{eqn3.18}), (\ref{eqn3.17}) yields
\begin{equation*}
  (L_1-2)h(Y,Z)=0,
\end{equation*}
which implies that $h(Y,Z)=0$ for all $Y,Z,$ on $M$, i.e., $M$ is totally geodesic, since $L_1\neq 2$.
The converse part is trivial. This proves the theorem.
\end{proof}
\begin{corollary}
  Let $M$ be a contact CR-submanifold of a Kenmotsu manifold $\overline{M}$ with respect to semisymmetric metric connection.
  Then $M$ is totally geodesic if and only if $M$ is semi parallel with respect  to semisymmetric metric connection.
\end{corollary}
From Corollary 4.1, Corollary 4.2, Theorem 4.1 and Theorem 4.2, we can state the following:
\begin{theorem}
  Let $M$ be a contact CR-submanifold of a Kenmotsu manifold $\overline{M}$ with respect to semisymmetric metric
   connection. Then the following statements are equivalent:\\
  \emph{(i)} $M$ is totally geodesic,\\
  \emph{(ii)} $M$ is parallel with respect to semisymmetric metric connection,\\
  \emph{(iii)} $M$ is semiparallel with respect to semisymmetric metric connection,\\
\emph{(iv)} $M$ is pseudo parallel in the sense of Chaki with respect to semisymmetric metric connection with $\alpha(\xi)\neq -2$,\\
\emph{(v)} $M$ is pseudo parallel in the sense of Deszcz with respect to semisymmetric metric connection with $L_1\neq 2$.
\end{theorem}
\section{conclusion}
In this paper pseudo parallel (in the sense of Chaki and in the sense of Deszcz) contact CR-submanifolds of Kenmotsu
manifolds are studied. It is known that pseudo Ricci symmetric manifolds or pseudo parallel manifolds in the sense of Chaki and Ricci
pseudosymmetric manifolds or pseudo parallel manifolds in the sense of Deszcz
are different. However, it is proved that pseudo parallel (in the sense of Chaki) contact CR-submanifolds and pseudo parallel (in the sense of Deszcz)
contact CR-submanifolds of Kenmotsu manifolds are equivalent with a certain condition. Also it is shown that pseudo parallel (in the sense of Chaki)
contact CR-submanifolds with respect to semisymmetric metric connection and pseudo parallel (in the sense of Deszcz) contact CR-submanifolds with respect
to semisymmetric metric connection of Kenmotsu manifolds with respect to semisymmetric metric connection are equivalent with a certain condition.\\

\noindent{\bf Acknowledgement:} The first author (S. K. Hui)  gratefully acknowledges to
the SERB (Project No.: EMR/2015/002302), Govt. of India for financial assistance of the work.

\vspace{0.1in}
\noindent Shyamal Kumar Hui and Pradip Mandal\\
Department of Mathematics\\
The University of Burdwan \\
Burdwan, 713104\\
 West Bengal, India\\
E-mail: skhui@math.buruniv.ac.in\\
E-mail:pradip2621994@rediffmail.com

\end{document}